\documentclass[12pt]{amsart}
\usepackage{amscd}
\usepackage{amssymb}
\newcommand{\be}{\begin{equation}}
\newcommand{\ee}{\end{equation}}
\newcommand{\ba}{\begin{eqnarray}}
\newcommand{\ea}{\end{eqnarray}}
\newcommand{\ban}{\begin{eqnarray*}}
	\newcommand{\ean}{\end{eqnarray*}}

\def\XXint#1#2#3{{\setbox0=\hbox{$#1{#2#3}{\int}$}
		\vcenter{\hbox{$#2#3$}}\kern-.5\wd0}}

\newcommand{\Rk}{\noindent {\bf Remark} }

\newtheorem{theo}{Theorem}[section]

\begin{document}
	\newtheorem{defn}[theo]{Definition}
	\newtheorem{ques}[theo]{Question}
	\newtheorem{lem}[theo]{Lemma}
	\newtheorem{prop}[theo]{Proposition}
	\newtheorem{coro}[theo]{Corollary}
	\newtheorem{ex}[theo]{Example}
	\newtheorem{note}[theo]{Note}
	\newtheorem{conj}[theo]{Conjecture}
	\makeatletter
	\@namedef{subjclassname@2020}{%
		\textup{2020} Mathematics Subject Classification}
	\makeatother

	\title[Ricci Shrinker with CSC]{a note on four dimensional  Shrinking Gradient Ricci Solitons with Constant Scalar Curvature }
	\author{Chen Wang}
	\address[Chen Wang]
	{School of Science, Zhejiang Sci-Tech University, Hangzhou 310018, China}
\email{2023210103029@mails.zstu.edu.cn}	
	\author{Guoqiang Wu}
	\address[Guoqiang Wu]
	{School of Science, Zhejiang Sci-Tech University, Hangzhou 310018, China}
	\email{gqwu@zstu.edu.cn}

	
	\date{}
	\maketitle

	\begin{abstract} Let $(M^4, g, f)$ be a four-dimensional complete noncompact gradient shrinking Ricci soliton with the equation $Ric+\nabla^2f= \frac{1}{2}g$.  If its scalar curvature is $1$,  Cheng-Zhou \cite{Cheng-Zhou} proved that it is a finite quotient of $\mathbb{R}^2\times \mathbb{S}^2$.
 In this  note we present an alternative proof by analyzing the asymptotic geometry at infinity.
	\end{abstract}
	
	\section{Introduction}
	Let  $(M^n, g)$ be an $n$-dimensional complete gradient Ricci soliton with the potential function $f$ satisfying
	\begin{align}\label{soliton}
	\text{Ric}+\nabla^2f=\lambda g
	\end{align}
	for some constant $\lambda$, where $\text{Ric}$ is the Ricci tensor of $g$ and $\nabla^2f$ denotes the Hessian of the potential function $f$.
	The Ricci soliton is said to be shrinking, steady, or expanding accordingly as $\lambda$ is positive, zero, or negative, respectively.
	\smallskip	
	
	A gradient Ricci soliton is a self-similar solution to the Ricci flow which flows by diffeomorphism and homothety. The study of solitons has become increasingly important in both the study of the Ricci flow introduced by Hamilton \cite{Hamilton} and metric measure theory. Solitons play a direct role as singularity dilations in the Ricci flow proof of uniformization. In \cite{Perelman1}, Perelman introduced the ancient $\kappa$-solutions, which
	play an important role in the singularity analysis, and he also proved that suitable blow down limit of ancient $\kappa$-solutions must be a shrinking gradient Ricci soliton. In \cite{Perelman2}, Perelman proved that any two dimensional non-flat ancient $\kappa$-soluition must be the standard $\mathbb S^2$, and he also classified three-dimensional shrinking gradient Ricci soliton under the assumption of nonnegative curvature and $\kappa$-noncollapseness. Due to the work of Perelman \cite{Perelman2}, Ni-Wallach \cite{Ni-Wallach}, Cao-Chen-Zhu \cite{Cao-Chen-Zhu}, the classification of three-dimensional shrinking gradient Ricci soliton is complete.

	\smallskip
	
	One particularly important case is shrinking gradient Ricci soliton with constant scalar curvature. About this direction, Professor Huai-Dong Cao conjectured that Ricci shrinker has constant scalar curvature if and only if it is isometric to a finite quotient of $\mathbb{R}^{n-k}\times \mathbb{N}^{k}$,  where $\mathbb{N}^k$ is an Einstein manifold with $Ric(g_\mathbb{N})=\frac{1}{2}g_\mathbb{N}$.
	
%
	\smallskip	
%
	
	\smallskip	
	Very recently, Cheng and Zhou \cite{Cheng-Zhou} confirmed Cao's conjecture in dimension $n=4$, together with works of Petersen-Wylie \cite{PW}
 and Fern\'{a}ndez-L\'{o}pez and Garc\'\i a-R\'\i o \cite{FR16}.
For four-dimensional complete shrinking gradient Ricci solitons with constant scalar curvature, by \cite{FR16} the possible values of R are $\{0, 2\lambda, 3\lambda,  4\lambda\}$. Moreover, if $R=0$ or $R= 4\lambda$, then the soliton is Einstein; If $R=3\lambda$, they are finite quotient of $\mathbb{S}^{3}\times\mathbb{R}$ \cite{FR16}.
	 Cheng and Zhou \cite{Cheng-Zhou} handled the most subtle case $R=(n-2)\lambda=2\lambda$. In this case, the Ricci curvature is non-negative by the work of Fern\'{a}ndez-L\'{o}pez and Garc\'\i a-R\'\i o \cite{FR16}.
	In addition, the Riemannian curvature is bounded since Munteanu-Wang  \cite{Munteaunu-Wang}  proved four-dimensional complete shrinking gradient Ricci solitons with bounded scalar curvature has bounded Riemannian curvature.
	However, for $n$-dimensional ($n\geq 5$), we cannot even guarantee the non-negativity of Ricci curvature and the boundedness of Riemannian curvature.
	Moreover, we point out that even if the $n$-dimensional ($n\geq 5$) shrinking gradient Ricci soliton has non-negative Ricci curvature and bounded Riemannian curvature, Cao's Conjecture is still open.
	
	\smallskip	

During the paper, we always assume that $\lambda=\frac{1}{2}$.

In Cheng-Zhou's work \cite{Cheng-Zhou}, 	they applied the weighted Laplacian $\Delta_f$ to the quantity $\textrm{tr}(\text{Ric}^3)$, the trace of the tensor $\textrm{Ric}^3$, for four-dimensional gradient shrinking Ricci soliton with constant scalar curvature $2\lambda$ and then derived the following nice inequality
	\ban
	\Delta_f  \left[ f(\textrm{tr}(\text{Ric}^3)-\frac14)\right] \geq 9 f\left[ \textrm{tr}(\text{Ric}^3)-\frac14\right].
	\ean
	Using integration by parts, they concluded that $\textrm{tr}(\text{Ric}^3)-\frac14=0$
	over $M$, 
	implying that the Ricci curvature has rank 2, and thus they obtained the rigidity result.

In this short note, we provide a new proof to  Cheng-Zhou's Theorem,  and the  main Theorem is stated as follows.
	\begin{theo}[\cite{Cheng-Zhou}]\label{main}
		Suppose $(M^4, g, f)$ is a four-dimensional shrinking gradient Ricci soliton with $R=1$, then it is isometric to a finite quotient of $\mathbb{R}^2 \times \mathbb{S}^2$.
	\end{theo}
	
	\Rk.   Our new proof is inspired by studying the asymptotic geometry at infinity. Naber \cite{Naber} proved that a shrinking gradient Ricci soliton $(M^n, g, f)$ converge along the integral curve of $f$ to $\mathbb{R}\times \mathbb{N}^{n-1}$, where $\mathbb{N}^{n-1}$ is an $(n-1)$-dimensional shrinking gradient Ricci soliton. In our case, $\mathbb{N}$ is of dimension three, and the three-dimensional shrinking gradient Ricci soliton is completely known. Another important new thing is that some geometric quantities can be calculated clearly under constant scalar curvature condition, which was unknown before.

%
%
%
%

	%
%
	\smallskip	


In Section 3, based on the point-picking argument, we prove the Riemannian curvature is bounded. With the curvature bound, we can   prove that $\lambda_1+\lambda_2 \rightarrow 0$.  In Section 4, we prove the key estimate of $|\nabla Ric|^2$,
	 and finish the proof of  Theorem \ref{main}.
	\smallskip

	
	\section{Notations and basic formulas on gradient shrinking Ricci solitons}\label{sec2}
	
	In this section, we recall the notations and basic formulas on gradient shrinking Ricci solitons with constant scalar curvature. For details, we refer to \cite{Cao,Cheng-Zhou,Hamilton,PW}.
	
	\smallskip
	Let  $(M, g)$ be an $n$-dimensional complete gradient shrinking Ricci soliton satisfying \eqref{soliton}.
	By scaling the metric $g$,  one can normalize $\lambda$ so that $\lambda=\frac{1}{2}$. In this paper, we always assume $\lambda=\frac{1}{2}$ and the gradient shrinking Ricci soliton equation is as follows,
	\begin{align}\label{soliton'}
	\text{Ric}+\nabla^2f=\frac{1}{2} g.
	\end{align}
	
	At first we recall some basic formulas which will be used throughout the paper:
	\begin{align}\label{second bianchi}
	d R=2 Ric(\nabla f,\cdot),
	\end{align}
	\begin{align}\label{R}
	R+\Delta f=\frac{n}{2},
	\end{align}\begin{align}\label{f he tidu}
	R+|\nabla f|^2=f,
	\end{align}\begin{align}\label{lr}
	\Delta_f R=R-2|Ric|^2,
	\end{align}\begin{align}\label{elliptic equation}
	\Delta_f R_{ij}=R_{ij}-2 R_{ikjl}R_{kl}.
	\end{align}
	where $\Delta _{f}=\Delta -\left\langle \nabla f,\nabla \right\rangle $ is the weighted Laplacian, and $\Delta _{f}$ acting on the function
	is self-adjoint on the space of square integrable functions with respect to the weighted measure $e^{-f}dv.$ In general, the weighted Laplacian $\Delta_f$ acting on tensors is given by $\Delta_f=\Delta-\nabla_{\nabla f}$.
	
	\smallskip

	\smallskip

	\begin{lem}\label{K1alafa} Let $(M^4, g, f)$ be a four-dimensional shrinking gradient Ricci soliton with constant scalar curvature $1$. Then
		\begin{equation}\label{k1a}
		K_{1\alpha}=\frac{\nabla_{\nabla f}R_{\alpha\alpha} +\lambda_\alpha(\frac{1}{2}-\lambda_\alpha)}{f}
		\end{equation}	
		for $\alpha=2,  3 ,4$.	
	\end{lem}
	
	\begin{proof}
		From the Ricci identity, we have
		\begin{equation*}
		\begin{aligned}
		&-R(\nabla f, e_\alpha, \nabla f, e_\beta) \\
		=&-\left( \nabla_\beta f_{\alpha k} - \nabla_kf_{\alpha \beta}\right) f_{k}\\
		=&\left( \nabla_\beta R_{\alpha k} - \nabla_kR_{\alpha \beta}\right) f_{k}\\
		=&-  \nabla_{\nabla f} R_{\alpha \beta}+\nabla_{\beta}(R_{\alpha k} f_k)- R_{\alpha k}f_{ k\beta}\\
		=&-  \nabla_{\nabla f} R_{\alpha \beta}- R_{\alpha k}\left( \frac{1}{2}g_{ k\beta}-R_{ k\beta} \right) \\
		=&- \nabla_{\nabla f} R_{\alpha \beta}-\left( \frac{1}{2}R_{\alpha \beta}-\sum_{k=1}^4 R_{\alpha k}R_{k\beta}\right),
		\end{aligned}
		\end{equation*}
		where \eqref{second bianchi} was used in the third equality. Therefore, we see
		\begin{equation}\label{R1a1b}
		R(e_1, e_\alpha, e_1, e_\beta)=\frac{ \nabla_{\nabla f} R_{\alpha \beta}+\left( \frac{1}{2}R_{\alpha \beta}-\sum_{k=1}^4 R_{\alpha k}R_{k\beta}\right) }{f}
		\end{equation}
		due to $|\nabla f|^2=f$. \eqref{k1a} holds by setting $\beta=\alpha$ in \eqref{R1a1b}.
		This completes the proof of the lemma.
	\end{proof}
	
\begin{lem}Suppose $(M^4, g, f)$ is a shrinking gradient Ricci soliton with $R=1$, then $Ric\geq 0$.
\end{lem}
\begin{proof}
This fact is known in the literature \cite{FR16}, for completeness, we give a proof.
Let $\lambda_1\leq \lambda_2\leq \lambda_3\leq \lambda_4$ be the eigenvalues of Ricci curvature. By Cauchy-Schwarz inequality,
\ban
(\lambda_3+\lambda_4)^2\leq 2(\lambda_3^2+\lambda_4^2),
\ean
since $\lambda_3+\lambda_4=1-\lambda_1-\lambda_2$ and $\lambda_3^3+\lambda_4^2=\frac{1}{2}-\lambda_1^2-\lambda_2^2$, we have
\ban
(1-\lambda_1-\lambda_2)^2\leq 2(\frac{1}{2}-\lambda_1^2-\lambda_2^2),
\ean
hence
\ban
2(\lambda_1^2+\lambda_2^2)+(\lambda_1+\lambda_2)^2\leq 2(\lambda_1+\lambda_2),
\ean
this implies that $\lambda_1+\lambda_2$ is always nonnegative. Combining with $Ric(\nabla f)=\frac{1}{2}dR=0$ gives the desired results.
\end{proof}

	\section{Curvature bound and uniform decay of $\lambda_1+\lambda_2$}\label{sec5}
	
	In this section, based on the point-picking argument, we will prove the Riemannian curvature is bounded.
	  By the similar argument, we can also prove that $\lambda_1+\lambda_2 \rightarrow 0$.
	
	\smallskip

In order to prove the curvature bound, we recall the following result.
\begin{lem}[\cite{Chen-Zhu}]\label{point picking}
	Given a complete noncompact Riemannian manifold with unbounded curvature, we can find a sequence of point $p_j$ divergent to infinity such that for each positive integer $j$, we have $|Rm(p_j)|\geq j$ and
	\ban
	|Rm(x)|\leq 4 |Rm(p_j)|
	\ean
	for $x\in B(p_j, \frac{j}{\sqrt{|Rm(p_j)|}})$.
\end{lem}

The following backward pseudolocality Theorem will also be useful.
\begin{theo}[\cite{Li-Wang2}]\label{backward pseudolocality} For any $\alpha>0$, there is $\epsilon(n, \alpha)$ such that the following holds.

 Let $(M^n, g(t))_{t\in I}$ be a Ricci flow induced by a shrinking gradient Ricci soliton. Given $(x_0, t_0)\in M\times I$ and $r>0$, if
 \ban
 |B_{t_0}(x_0, r)|\geq \alpha r^n, \quad |Rm|\leq (\alpha r)^{-2} \quad on \quad B_{t_0}(x_0, r),
 \ean
 then
 \ban
 |Rm|\leq (\epsilon r)^{-2}  \quad on \quad P(x_0, t_0; (1-\alpha)r, -(\epsilon r)^2, 0).
 \ean
\end{theo}

\begin{theo}\label{bounded curvature}Suppose $(M^4, g, f)$ is a four-dimensional shrinking gradient Ricci soliton with $R=1$, then its curvature is bounded.
\end{theo}

\Rk.   Munteanu-Wang \cite{Munteaunu-Wang} proved that four-dimensional Ricci shrinker has bounded curvature if its scalar curvature is bounded. Here in this special case we can give a simple proof.
\begin{proof}
%
%
Suppose not, by Lemma \ref{point picking}, then there exists a sequence of points $\{p_j\}_{j=1}^\infty$ divergent to infinity such that for each positive integer $j$, we have $|Rm(p_j)|\geq j$ and
	\ban
	|Rm(x)|\leq 4 |Rm(p_j)|
	\ean
	for $x\in B(p_j, \frac{j}{\sqrt{|Rm(p_j)|}})$.
	
	By the $\kappa$ noncollapsed theorem in \cite{Li-Wang} and the scalar curvature is bounded for $(M, g, f)$, we get that $\text{Vol}(B(p_j, r))\geq \kappa r^4$ for $0<r<1$ and some  $\kappa$ depending only on the soliton.

For the shrinking gradient Ricci soliton  $(M,g,f)$,
 the associated Ricci flow $(M, g(t), p_j)$ is defined on $(-\infty, 0)$, where
\[
g(t):=(-t)\phi^*_tg,\quad \frac{d\phi_t}{dt}=\frac{\nabla f}{-t}
\quad  \text{and}\quad \phi_{-1}=\mathrm{Id}.
\]
Now we can use Theorem \ref{backward pseudolocality} with $r=|Rm|(p_j)^{-\frac{1}{2}}$ to derive that
\ban
|Rm|\leq (\epsilon r)^{-2} \quad on \quad B(p_j, g(-1), \frac{j}{2}r) \times [-1-(\epsilon r)^2, -1].
\ean
Then we can apply Hamilton's compactness theorem to obtain that the rescaled manifolds $\left(B(p_j, g, \frac{j}{2\sqrt{|Rm(p_j)|}}),  |Rm(p_j)|g, p_j\right)$ converge to a smooth complete Riemannian manifold $(M_\infty^4, g_\infty, p_\infty)$ with $|Rm(p_\infty)|=1$ which is Ricci flat because $(M^4, g)$ has bounded Ricci curvature and $|Rm(p_j)|\rightarrow \infty$, moreover  $(M_\infty^4, g_\infty, p_\infty)$ has Euclidean volume growth.
	
	\smallskip
	Since the integral curves of $f$ passing through $p_j$ is a geodesic with respect to $(M, g)$, the geodesic segment of these curves contained in $B\left( p_j, g,  \frac{j}{2\sqrt{|Rm(p_j)|}}\right) $ will converge to a geodesic line in $(M_\infty, g_\infty)$, then Cheeger-Gromoll's splitting theorem \cite{Cheeger-Gromoll} implies that $M_\infty^4 =\mathbb{R}\times N^3$, where $N^3$ is Ricci flat, hence flat. This contradicts with $|Rm|(p_\infty)=1$.
\end{proof}

\begin{theo}\label{uniform decay}
	Suppose $(M^4, g, f)$ is a four-dimensional shrinking gradient Ricci soliton with $R= 1$, then $\lambda_1+\lambda_2\rightarrow 0$ at infinity.
\end{theo}

\begin{proof}
Suppose on the contrary, then there is a sequence of points
$q_i\to\infty$ but $(\lambda_1+\lambda_2)(q_i)\geq \delta$ for some $\delta>0$.
For   four-dimensional shrinker $(M,g,f)$ with
bounded curvature, the associated Ricci flow $(M, g(t), q_i)$ defined on $(-\infty, 0)$, where
\[
g(t):=(-t)\phi^*_tg,\quad \frac{d\phi_t}{dt}=\frac{\nabla f}{-t}
\quad  \text{and}\quad \phi_{-1}=\mathrm{Id},
\]
is $\kappa$-noncollapsed, where $\kappa=\kappa(n, V_f(M))$. By Hamilton's
Cheeger-Gromov compactness theorem, the associated Ricci flow
sub-converges to an ancient $\kappa$-noncollapsed solution
$(M_{\infty}, g_{\infty}(t), q_{\infty})$.  Consider a sequence of functions
\[
f_i(x):=\frac{f(x)-f(q_i)}{|\nabla f(q_i)|},
\]
which satisfies
\[
|\nabla f_i(q_i)|=1 \quad\text{and} \quad
|\mathrm{Hess}\,f_i|=\frac{|\tfrac{1}{2}g-Ric(g)|}{|\nabla f(q_i)|}
\]
on $(M^n,g(t))$. Since $|\nabla f(x)|\to \infty$ as $x\to\infty$,
$|\nabla f(q_i)|\to\infty$ as $q_i\to\infty$. Combining this with the fact that the curvature is bounded, we deduce that $|\mathrm{Hess}\, f_i|\to 0$ uniformly as
$i\to\infty$. Hence, along the convergence of $(M^n, g(t), q_i)$, the sequence
of functions $f_i(x)$ smoothly converges to a limit $f_\infty$ satisfying
$|\nabla f_\infty(q_{\infty})|=1$ and $\mathrm{Hess}\, f_\infty\equiv 0$
on $(M_{\infty}, g_{\infty}(t))$. This implies that
$(M_{\infty}, g_{\infty}(t), q_{\infty})$ isometrically splits
as $(\mathbb{R}\times \mathbb{N}^3, g_\mathbb{R}+g_{\mathbb{N}^3(t)})$, where $(\mathbb{N}^3, g_{\mathbb{N}^3(t)})$
is a three-dimensional ancient  $\kappa$-solution with $R=1$ at $t=-1$.
By the recent works of  \cite{Angenent-Brendle-Daskalopoulos-Sesum,Brendle,Brendle-Daskalopoulos-Sesum}
and \cite{Bamler-Kleiner} for the compact case, there are four possibilities for $(\mathbb{N}^3, g_{\mathbb{N}^3(t)})$.

(1).  $\mathbb{N}^3$ is isometric to $\mathbb{S}^3/\Gamma$ with $R=1$ at time $t=-1$. Hence the corresponding Ricci flow solution is $g_{\mathbb{N}^3(t)}=(2-4t)g_{\mathbb{S}^3}$, which contradicts the fact that the time span of $g_{\mathbb{N}^3(t)}$ should be $(-\infty, 0)$ from the convergence process.

(2).  $\mathbb{N}^3$ is isometric to a finite quotient of the Type II ancient solution
constructed by Perelman, this contradicts $R=1$.

(3).  $\mathbb{N}^3$ is isometric to the Bryant soliton, this also contradicts  $R=1$.

(4).  $\mathbb{N}^3$ is isometric to $(\mathbb{R}\times \mathbb{S}^2)/\Gamma$,
 then $(\lambda_1+\lambda_2)(q_\infty)=0$.

In conclusion, $(\lambda_1+\lambda_2)\to 0$ uniformly at infinity.

\end{proof}

%
%

\section{key estimate of $|\nabla Ric|^2$}
In this section, we establish a key estimate for $|\nabla Ric|^2$ on a $4$-dimensional shrinking gradient Ricci soliton with $R=1$.  More precisely, we have the following theorem.

\begin{theo}Suppose $(M^4, g, f)$ is a shrinking gradient Ricci soliton with $R=1$, then $|\nabla Ric|^2\leq -0.9(\lambda_1+\lambda_2)+K_{12}$  on $M\setminus D(a)$ for some large $a>0$.
\end{theo}

\begin{proof}
For  simplicity, throughout the proof, we denote $K_{ij}=R(e_i, e_j, e_i, e_j)$.
Following similar argument as Theorem \ref{uniform decay}, one sees  that $K_{ij}\rightarrow 0$ except for $K_{34}$. By $(2.7)$, we have
\ban
\frac{1}{2}\Delta_f |Ric|^2=|Ric|^2+|\nabla Ric|^2-2K_{ij}\lambda_i\lambda_j.
\ean
This jointly with $(2.6)$ yields
\ban
&&|\nabla Ric|^2\\
&&=2K_{ij}\lambda_i\lambda_j-|Ric|^2\\
&&=4K_{23}\lambda_2\lambda_3+4K_{24}\lambda_2\lambda_4+4K_{34}\lambda_3\lambda_4-\frac{1}{2}\\
&&=o(1)(\lambda_1+\lambda_2)+4K_{34}\lambda_3\lambda_4-\frac{1}{2}
\ean
Since
\ban
&& \lambda_1=K_{12}+K_{13}+K_{14}\\
&&\lambda_2=K_{21}+K_{23}+K_{24}\\
&&\lambda_3=K_{31}+K_{32}+K_{34}\\
&&\lambda_4=K_{41}+K_{42}+K_{43},
\ean
we get
\ban
2K_{34}&&=2K_{12}+\lambda_3+\lambda_4-\lambda_1-\lambda_2\\
&&=2K_{12}+1-2(\lambda_1+\lambda_2).
\ean
Combined with $\lambda_3\lambda_4=(\lambda_3-\frac{1}{2})(\lambda_4-\frac{1}{2})+\frac{1}{2}(\lambda_3+\lambda_4)-\frac{1}{4}$, we have
\ban
&&|\nabla Ric|^2\\
&&=o(1)(\lambda_1+\lambda_2)+2\left(2K_{12}+1-2(\lambda_1+\lambda_2)\right)\\
&&\quad \times \left((\lambda_3-\frac{1}{2})(\lambda_4-\frac{1}{2})+\frac{1}{2}(\lambda_3+\lambda_4)-\frac{1}{4}\right)-\frac{1}{2}.
\ean
Notice that
\ban
&&2(\lambda_3-\frac{1}{2})(\lambda_4-\frac{1}{2})\\
&&\leq (\lambda_3-\frac{1}{2})^2+(\lambda_4-\frac{1}{2})^2\\
&&= \lambda_3^2+\lambda_4^2-(\lambda_3+\lambda_4)+\frac{1}{2}\\
&&=\frac{1}{2}-\lambda_1^2-\lambda_2^2-(1-\lambda_1-\lambda_2)-\frac{1}{2}\\
&&\leq \lambda_1+\lambda_2.
\ean
then we can obtain
\ban
&&|\nabla Ric|^2\\
&&=o(1)(\lambda_1+\lambda_2)+(2|K_{12}|+1-2(\lambda_1+\lambda_2))(\lambda_1+\lambda_2)\\
&&\quad (2K_{12}+1-2(\lambda_1+\lambda_2))(1-(\lambda_1+\lambda_2))-\frac{1}{2}(2K_{12}+1-2(\lambda_1+\lambda_2))-\frac{1}{2}\\
&&=o(1)(\lambda_1+\lambda_2)-(\lambda_1+\lambda_2)+K_{12}.
\ean
So $|\nabla Ric|^2\leq -0.9(\lambda_1+\lambda_2)+K_{12}$  on $M\setminus D(a)$ for some large $a>0$.

\end{proof}
As a consequence of Theorem $4.1$, we get the following corollary.
\begin{coro}\label{to be used}Suppose $(M^4, g, f)$ is a shrinking gradient Ricci soliton with $R=1$, then
 \ban
&&\int_{\Sigma(s)}|\nabla Ric|^2\\
&&\leq  -0.8\int_{\Sigma(s)}(\lambda_1+\lambda_2)+\frac{1}{s}\int_{\Sigma(s)} \nabla f \cdot\nabla(\lambda_1+\lambda_2)
\ean
for almost everywhere $s\geq a$, where $a>0$ is large enough.
\end{coro}
\Rk.    $\lambda_1+\lambda_2$   is only Lipschitz continuous and differentiable almost everywhere.                             Here                    $\langle\nabla (\lambda_1+\lambda_2), \nabla f \rangle$ can also be understood as follows: At any $p\in M$, choose $\{e_1, e_2\}$ such that $\{e_1, e_2\}$ are the eigenvectors corresponding to the eigenvalues $\{\lambda_1, \lambda_2\}$, where $\lambda_1\leq  \lambda_2$ are the smallest two eigenvalues of Ricci curvature, take parallel  transport along all the geodesics starting from $p$, then we obtain two smooth vector fields $\{e_1, e_2\}$  in a neighborhood of $p$ such that $e_1(p)=e_1$, $e_2(p)=e_2$. It is easy to check that $(\nabla_{\nabla f}Ric)(e_1, e_1)+(\nabla_{\nabla f}Ric)(e_2, e_2)=\nabla f\cdot \left(Ric(e_1, e_1)+Ric(e_2, e_2)\right) =\langle\nabla f ,  (\lambda_1+\lambda_2)\rangle$ if $\lambda_1+\lambda_2$ is differentiable at $p$.
\begin{proof}
By Lemma \ref{K1alafa}, we know $K_{12}=\frac{\nabla f\cdot\nabla(\lambda_1+\lambda_2)}{f}+\frac{\lambda_2(\frac{1}{2}-\lambda_2)}{f}$.
hence
\ban
&&\int_{\Sigma(s)}|\nabla Ric|^2\\
&&\leq -0.9\int_{\Sigma(s)}(\lambda_1+\lambda_2)+\frac{1}{s}\int_{\Sigma(s)} \nabla f \cdot\nabla(\lambda_1+\lambda_2)+\frac{1}{s}\int_{\Sigma(s)}\lambda_2(\frac{1}{2}-\lambda_2)\\
&&\leq  -0.8\int_{\Sigma(s)}(\lambda_1+\lambda_2)+\frac{1}{s}\int_{\Sigma(s)} \nabla f \cdot\nabla(\lambda_1+\lambda_2)
\ean
for $s\geq a$, where $a>0$ is large enough.

\end{proof}

\begin{prop}\label{other direction}
	Let $(M^4, g, f)$ be a four-dimensional shrinking gradient Ricci soliton with $R=1$.  then
\ban
\int_{\Sigma(s)} \langle\nabla (\lambda_1+\lambda_2), \nabla f \rangle d\sigma_{\Sigma(s)}\leq 0
\ean
	for sufficiently large almost everywhere $s$.
\end{prop}

\begin{proof}
	For the purpose, we consider the following one parameter of diffeomorphisms,
	\begin{equation*}
		\begin{aligned}
			\begin{cases}
				&\frac{\partial F}{\partial s}=\frac{\nabla f}{|\nabla f|^2},\\
				&F(x, a)=x \in \Sigma(a).
			\end{cases}
		\end{aligned}
	\end{equation*}
	Then $\frac{\partial f}{\partial s}=\langle\nabla f, \frac{\nabla f}{|\nabla f|^2} \rangle=1$,
	and the advantage of $F$ is that it maps level set of $f$ to another level set, in particular $f(F(x, s))=s$.
	
	\smallskip
	Suppose $\{x_1, x_2, x_3\}$ are local coordinate chart of $\Sigma(a)$, on $\Sigma(s)$, let  $g(s)(\frac{\partial }{\partial x_i}, \frac{\partial}{\partial x_j}):=g(\frac{\partial F}{\partial x_i}, \frac{\partial F}{\partial x_j})$, $d\sigma_{\Sigma(s)}=\sqrt{det(g_{ij})}dx$, where $dx=dx_1\wedge dx_2\wedge dx_3$.
	Next we compute the derivatives of $d\sigma_{\Sigma(s)}$.
	\begin{align*}
		&\frac{\partial}{\partial s}d\sigma_{\Sigma(s)}=\frac{\partial}{\partial s}\sqrt{det(g_{ij})}dx\\
		=&\frac{1}{2}\cdot 2 g^{ij}\langle \nabla_{\frac{\partial F}{\partial x_i}}\frac{\partial F}{\partial s}, \frac{\partial F}{\partial x_j}\rangle d\sigma_{\Sigma(s)}\\
		=& g^{ij}\langle \nabla_{\frac{\partial F}{\partial x_i}}\frac{\nabla f}{|\nabla f|^2}, \frac{\partial F}{\partial x_j}\rangle d\sigma_{\Sigma(s)}\\
		=&\frac{1}{|\nabla f|^2}g^{ij}\nabla^2 f(\frac{\partial F}{\partial x_i}, \frac{\partial F}{\partial x_j})d\sigma_{\Sigma(s)}\\
		=&\frac{1}{|\nabla f|^2} g^{ij}\left(\frac{1}{2}g(\frac{\partial F}{\partial x_i}, \frac{\partial F}{\partial x_j})-Ric(\frac{\partial F}{\partial x_i}, \frac{\partial F}{\partial x_j})\right)d\sigma_{\Sigma(s)}\\
		=&\frac{1}{|\nabla f|^2} (\frac{3}{2}-\frac{2}{2})d\sigma_{\Sigma(s)}=\frac{1}{2s}d\sigma_{\Sigma(s)}.
	\end{align*}
	Hence, it is easy to check that
	\ban
	\frac{\partial}{\partial s}\left(\frac{1}{\sqrt{s}}d\sigma_{\Sigma(s)}\right)=\left(-\frac{1}{2}s^{-\frac{3}{2}}+\frac{1}{\sqrt{s}}\frac{1}{2s}\right) d\sigma_{\Sigma(s)}=0.
	\ean
	Define a function
	\ban
	I(s)=\int_{\Sigma(s)} (\lambda_1+\lambda_2)\cdot \frac{1}{\sqrt{s}} d\sigma_{\Sigma(s)},
	\ean
since $\lambda_1+\lambda_2$  is Lipschitz continuous,
$I(s)$ is also Lipschitz continuous, hence differentiable almost everywhere,
and then we can compute the derivative of $I(s)$ as follows:
	\ban
	I'(s)&&=\frac{d}{ds}\int_{\Sigma(s)} (\lambda_1+\lambda_2)\cdot \frac{1}{\sqrt{s}} d\sigma_{\Sigma(s)}\\
	&&=\int_{\Sigma(s)} \langle \nabla (\lambda_1+\lambda_2), \frac{\nabla f}{|\nabla f|^2}\rangle \frac{1}{\sqrt{s}}d\sigma_{\Sigma(s)}\\
	&&+\int_{\Sigma(s)} (\lambda_1+\lambda_2)\frac{\partial}{\partial s}\left(\frac{1}{\sqrt{s}}d\sigma_{\Sigma(s)}\right)\\
	&&=\int_{\Sigma(s)} \langle \nabla (\lambda_1+\lambda_2), \frac{\nabla f}{|\nabla f|^2}\rangle \frac{1}{\sqrt{s}}d\sigma_{\Sigma(s)}\\
	&&=\frac{1}{s^{\frac{3}{2}}}\int_{\Sigma(s)} \langle \nabla (\lambda_1+\lambda_2), \nabla f\rangle d\sigma_{\Sigma(s)},
	\ean
	where we have  used  $|\nabla f|^2=s$  in the last equality.
	
	\smallskip
	Moreover, since $I(s)$ tends to zero as $s\rightarrow\infty$ by Theorem \ref{uniform decay}, there exists a sufficiently large $b>a$ such that $I'(b)\leq 0$, i.e.
	\ban
	\int_{\Sigma(b)} \langle \nabla(\lambda_1+\lambda_2),\nabla f \rangle d\sigma_{\Sigma(b)}\leq 0.
	\ean
	
Finally, we claim that
\ban
\int_{\Sigma(s)} \langle\nabla (\lambda_1+\lambda_2), \nabla f \rangle d\sigma_{\Sigma(s)}\leq 0
\ean
for almost everywhere $s$ with $s\geq b$.
In fact, if not, assume there is some $c>b$, such that
\ban
\int_{\Sigma(c)} \langle\nabla (\lambda_1+\lambda_2), \nabla f \rangle d\sigma_{\Sigma(c)}> 0.
\ean
Similarly, because $I(s)$ tends to zero as $s\rightarrow\infty$, there exists sufficiently large $d>c$ such that $I'(d)< 0$, i.e.,
	\ban
	\int_{\Sigma(d)} \langle\nabla (\lambda_1+\lambda_2), \nabla f \rangle d\sigma_{\Sigma(d)}< 0.
	\ean
Then it follows from Corollary \ref{to be used} that
	\ban
&&\int_{\Sigma(d)}|\nabla Ric|^2 d\sigma_{\Sigma(d)}\\
&\leq&
-0.8\int_{\Sigma(d)} (\lambda_1+\lambda_2)\,  d\sigma_{\Sigma(d)}
+ \frac{1}{d}\int_{\Sigma(d)} \langle\nabla (\lambda_1+\lambda_2), \nabla f \rangle d\sigma_{\Sigma(d)}\\
&<&0,
\ean	
which is a contradiction.	
So, We finish  the proof of Proposition \ref{other direction}.
\end{proof}

Now we can finish the proof of Theorem \ref{main}.

\begin{proof}
We  apply Corollary \ref{to be used} and Proposition  \ref{other direction} to infer
\ban
\int_{\Sigma(s)}|\nabla Ric|^2 d\sigma_{\Sigma(s)}=0 \,\,\, \text{and}
\int_{\Sigma(s)} (\lambda_1+\lambda_2)\, d\sigma_{\Sigma(s)}=0
\ean
for sufficiently large almost everywhere $s$ due to the nonnegativity of $\lambda_1+\lambda_2$. Thus  $\nabla Ric =0$ and $\lambda_1+\lambda_2=0$  on $M \setminus D(s)$ by the continuity of $\nabla Ric$ and $\lambda_1+\lambda_2$.
Since Ricci curvature is nonnegative, we get
 \ban
 \lambda_1=\lambda_2\equiv 0 \,\,\, \text{and}\,\,\, \lambda_3=\lambda_4\equiv \frac{1}{2}.
 \ean
Due to the analyticity of gradient Ricci soliton,  $\nabla Ric=0$ on $M$.
 Finally, De Rham's splitting theorem implies that $(M^4, g, f)$ is isometric to a finite quotient of  $\mathbb{R}^2\times\mathbb{ N}^2$, where $\mathbb{N}^2$ is a two-dimensional Einstein manifold with Einstein constant $\frac{1}{2}$, has to be isometric to $\mathbb{S}^2$.
We have completed the proof of Theorem \ref{main}.

\end{proof}

\end{document}